\documentclass[12pt]{amsart}
\usepackage[colorlinks=true,citecolor=black,linkcolor=black,urlcolor=blue]{hyperref}
\usepackage{amsmath}
\usepackage[english,  activeacute]{babel}
\usepackage[utf8]{inputenc}
\usepackage{amssymb}
\usepackage{amsthm}
\usepackage{graphics,graphicx}
\usepackage{enumerate}
\usepackage{array}
\usepackage{bm}
\usepackage{a4wide}
\usepackage{color, url}
\usepackage{float}
\usepackage{textcomp}

 \newtheorem{thm}{Theorem}[section]
 \newtheorem{cor}[thm]{Corollary}
 
 \newtheorem{prop}[thm]{Proposition}
 \theoremstyle{definition}
 \newtheorem{defn}[thm]{Definition}
 \theoremstyle{remark}
 \newtheorem{rem}[thm]{Remark}
 \newtheorem*{ex}{Example}
 \numberwithin{equation}{section}

\theoremstyle{plain}

\title[The $(l,r)$-Stirling numbers: a combinatorial approach.]{The $(l,r)$-Stirling numbers: a combinatorial approach.}

\author[BELBACHIR H.]{Belbachir Hac\`ene}

\address{%
USTHB, Faculty of Mathematics, RECITS Laboratory\\
P.O. Box 32\\ 
El Alia, 16111, Bab Ezzouar, Algiers \\ 
Algeria}

\email{hacenebelbachir@gmail.com, hbelbachir@usthb.dz}

\author[DJEMMADA Y.]{Djemmada Yahia}
\address{%
USTHB, Faculty of Mathematics, RECITS Laboratory\\
P.O. Box 32\\ 
El Alia, 16111, Bab Ezzouar, Algiers \\ 
Algeria}
\email{yahia.djem@gmail.com, ydjemmada@usthb.dz}

\date{\today}


\begin{document}
\begin{abstract}
  This work deals with a new generalization of $r$-Stirling numbers using $l$-tuple of permutations and partitions called $(l,r)$-Stirling numbers of both kinds. We study various properties of these numbers using combinatorial interpretations and symmetric functions. Also, we give a limit representation of the multiple zeta function using $(l,r)$-Stirling of the first kind.
\end{abstract}
\maketitle

\section{Introduction}

Let $\sigma$ be a permutation of the set $[n]=\{1,2,\dots,n\}$ having $k$ cycles $c_1,c_2,\dots,c_k$. A \textit{cycle leaders set} of $\sigma$, denoted $cl(\sigma)$, is the set of the smallest elements on their cycles, i. e.  
\begin{equation*}
cl(\sigma)=\{\min c_1, \min c_2, \dots ,\min c_k\}.
\end{equation*}

As the same way, let $\pi$ be a partition of the set $[n]=\{1,2,\dots,n\}$ into $k$ blocks $b_1,b_2,\dots,b_k$. A \textit{block leaders set} of $\pi$, denoted $bl(\pi)$, is the set of the smallest elements on their blocks,  i. e.
\begin{equation*}
bl(\pi)=\{\min b_1, \min b_2, \dots ,\min b_k\}.
\end{equation*}
\begin{ex}
\begin{itemize}
\item[]
\item For $n=6$, the permutation $\sigma=(13)(245)(6)$ have the set of cycle leaders $cl(\sigma)=\{1,2,6\}$.
\item For $n=7$, the partition $\pi=1,2,4|3,5,7|6$ have the set of block leaders $bl(\pi)=\{1,3,6\}$.
\end{itemize}
\end{ex}

It is well known that the \textbf{Stirling numbers of the first kind}, denoted $ {n \brack k}$, count the number of all permutations of $[n]$ having exactly $k$ cycles, and \textbf{Stirling numbers of the second kind}, denoted $ {n \brace k}$, count the number of all partitions of $[n]$ having exactly $k$ blocks.

One of the most interesting generalization of Stirling numbers was the $r$-Stirling numbers of both kind introduced By Broder \cite{Bro}. Analogously to the classical Stirling numbers of both kinds, the author considered that $r$-Stirling numbers of the first kind ${n \brack k}_r$ (resp. the second kind ${n \brace k}_r$) counts the number of permutations $\sigma$ (resp. partitions $\pi$) having exactly $k$ cycles (resp. $k$ blocks) such that the $r$ first elements $1,2,\dots,r$ lead.

Dumont, in \cite{Dum}, gives the first interpretation for the "central factorial" numbers of the second kind $U(n,k)$ given by the recurrence
\begin{equation}
U(n,k)=U(n-1,k-1)+k^2 U(n-1,k), \qquad \text{for } 0<k\leq n;
\end{equation}
where $U(n,k)=T(2n,2k)$.

Then, using the notion of \textit{quasi-permutations}, Foata and Han \cite{Foa} , showed that $U(n,k)$
counts the number of pair $(\pi_1,\pi_2)$-partitions of $[n]$ into $k$ blocks such that $bl(\pi_1)=bl(\pi_2)$.

In this work, we give an extension of the $r$-Stirling numbers of both kinds with considering $l$-tuple partitions (and permutations) of Dumont's partition model \cite{Dum,Foa}.

This paper is organized as follows. In Section \ref{Sec2} and Section \ref{Sec3}, we introduce the $(l,r)$-Stirling numbers of both kinds. Some properties are given as recurrences, orthogonality, generating functions, a relation between  $(l,r)$-Stirling numbers and Bernoulli polynomials via Faulhaber sums and symmetric functions. In Section \ref{Sec4}, we show the relations between multiple-zeta function and the $(l,r)$-Stirling numbers. Finally, in Section \ref{Sec6}, we discuss some remarks which connect this numbers to the rooks polynomials \cite{Kr}.

\section{The $(l,r)$-Stirling numbers of both kinds}\label{Sec2}
Let us consider the following generalization,
\begin{defn}The $(l,r)$-Stirling number of the first kind ${n \brack k}_r^{(l)}$ counts the number of $l$-tuple of permutations $(\sigma_1,\sigma_2,\dots,\sigma_l)$ of $[n]$ having exactly $k$ cycles such that  $1,2, \dots,r$ first elements lead, and  $$cl(\sigma_1)=cl(\sigma_2)=\dots=cl(\sigma_l).$$
\end{defn}

\begin{defn}The $(l,r)$-Stirling number of the second kind ${n \brace k}_r^{(l)}$ counts the number of $l$-tuple of partitions $(\pi_1,\pi_2,\dots,\pi_l)$ of $[n]$ having exactly $k$ blocks such that  $1,2, \dots,r$ first elements lead, and  $$bl(\pi_1)=bl(\pi_2)=\dots=bl(\pi_l).$$
\end{defn}

\begin{thm}\label{RecS1} The $(l,r)$-Stirling numbers of the first satisfy the following recurrences
\begin{equation}
{n \brack k}_r^{(l)}={n-1 \brack k-1}_r^{(l)}+(n-1)^l{n-1 \brack k}_r^{(l)}, \qquad \text{for }  n >  r
\end{equation}
and
\begin{equation}
{n \brack k}_r^{(l)}=\frac{1}{(r-1)^l}\left({n \brack k-1}_{r-1}^{(l)}- {n \brack k-1}_r^{(l)}\right), \qquad \text{for }  n \geq r > 1.
\end{equation}
with boundary conditions ${n \brack k}_r^{(l)} = 0,$ for $n<r$; and ${n \brack k}_r^{(l)} = \delta_{k,r}$, for $n=r$.
\end{thm}
\begin{proof}

The $(\sigma_1,\sigma_2,\dots,\sigma_l)$-permutations of the set $[n]$ having $k$ cycles such that $1,2,\dots,r$ first elements are in distinct cycles and $cl(\sigma_1)=cl(\sigma_2)=\dots=cl(\sigma_l)$ is either obtained from:
\begin{itemize}
\item Inserting the $nth$ elements after any element in each permutation of $(\sigma_1,\sigma_2,\dots,\sigma_l)$-permutations of the set $[n-1]$ having $k$ cycles such that $1,2,\dots,r$ first elements are in distinct cycles and $cl(\sigma_1)=cl(\sigma_2)=\dots=cl(\sigma_l)$, hence there are $(n-1)^l {n-1 \brack k}_r^{(l)}$ choices.
\item The $nth$ element forms a cycle in each permutation of $(\sigma_1,\sigma_2,\dots,\sigma_l)$-permutations, the remaining $[n-1]$ have to be $(\sigma_1,\sigma_2,\dots,\sigma_l)$-permuted in $(k-1)$ cycles  under the preceding conditions, hence there are ${n-1 \brack k-1}_r^{(l)}$.
\end{itemize}

This correspondence yields the first recurrence.

For the second recurrence, we use the double counting principle. Let us count the numbers of $(\sigma_1,\sigma_2,\dots,\sigma_l)$-permutations of the set $[n]$ having $(k-1)$ cycles such that $1,\dots,r-1$ are cycle leaders but $r$ is not, with $cl(\sigma_1)=cl(\sigma_2)=\dots=cl(\sigma_l)$, this is either obtained from: 
\begin{itemize}
\item We count the $(\sigma_1,\sigma_2,\dots,\sigma_l)$-permutations of the set $[n]$ having $(k-1)$ cycles such that $1,\dots,r-1$ are cycle leaders then we exclude from them the $(\sigma_1,\sigma_2,\dots,\sigma_l)$-permutations having $r$ as cycle leader. That gives $${n \brack k-1}_{r-1}^{(l)}- {n \brack k-1}_r^{(l)},$$
\item Or we count the $(\sigma_1,\sigma_2,\dots,\sigma_l)$-permutations of the set $[n]$ having $k$ cycles such that $1,\dots,r$ are cycle leaders then we appending the cycle having $r$ as leader at the end of a cycle having a smaller leader. We have $(r-1)$ choices to do in each permutation. That gives 
$$(r-1)^l{n \brack k}_r^{(l)},$$
\end{itemize}
from the two ways of counting we get the result. 
\end{proof}
\begin{thm}\label{RecS2}The $(l,r)$-Stirling numbers of the second satisfy the following recurrences 
\begin{equation}\label{IdB1}
{n \brace k}_r^{(l)}={n-1 \brace k-1}_r^{(l)}+k^l{n-1 \brace k}_r^{(l)}, \qquad \text{for }  n >  r
\end{equation}
and
\begin{equation}\label{IdB2}
{n \brace k}_r^{(l)}={n \brace k}_{r-1}^{(l)}- (r-1)^l{n-1 \brace k}_{r-1}^{(l)}, \qquad \text{for }  n \geq r > 1.
\end{equation}
with boundary conditions ${n \brace k}_r^{(l)} = 0$, for $n<r$; and ${n \brace k}_r^{(l)} = \delta_{k,r}$, for $n=r$.
\end{thm}
\begin{proof}
As in Theorem \ref{RecS1}, the $(\pi_1,\pi_2,\dots,\pi_l)$-partitions of the set $[n]$ into $k$ blocks such that $1,2,\dots,r$ first elements are in distinct blocks and $bl(\pi_1)=bl(\pi_2)=\dots=bl(\pi_l)$ is either obtained from:
\begin{itemize}
\item Inserting the $nth$ elements in a block of each partition of $(\pi_1,\pi_2,\dots,\pi_l)$-partitions of the set $[n-1]$ into $k$ blocks such that $1,2,\dots,r$ first elements are in distinct blocks and $bl(\pi_1)=bl(\pi_2)=\dots=bl(\pi_l)$, hence there are $k^l {n-1 \brace k}_r^{(l)}$ choices (the position of the $nth$ element in a block doesn't matter).

\item The $nth$ element forms a block in each partition of $(\pi_1,\pi_2,\dots,\pi_l)$-partitions, the remaining $[n-1]$ have to be $(\pi_1,\pi_2,\dots,\pi_l)$-partitioned into $(k-1)$ blocks  under the preceding conditions, hence there are ${n-1 \brace k-1}_r^{(l)}$.
\end{itemize}

For the Identity \eqref{IdB2}, we use the double counting principle to count the numbers of $(\pi_1,\pi_2,\dots,\pi_l)$-partitions of $[n]$ into $k$ blocks such that $1,2,\dots,(r-1)$ are block leaders but $r$ is not, with $bl(\pi_1)=bl(\pi_2)=\dots=bl(\pi_l)$, this is either obtained from:
\begin{itemize}
\item We count the $(\pi_1,\pi_2,\dots,\pi_l)$-partitions of the set $[n]$ into $k$ blocks such that $1,\dots,r-1$ are block leaders then we exclude from them the $(\pi_1,\pi_2,\dots,\pi_l)$-partitions having $r$ as block leader, with $bl(\pi_1)=bl(\pi_2)=\dots=bl(\pi_l)$. That gives $${n \brace k}_{r-1}^{(l)}- {n \brace k}_r^{(l)},$$
\item Or we count the $(\pi_1,\pi_2,\dots,\pi_l)$-partitions of the set $[n]$\textbackslash$\{r\}$ into $k$ blocks such that $1,\dots,r-1$ are block leaders then we include the  element $\{r\}$ in any block having a smaller leader then $r$. We have $(r-1)$ choices to do in each partition of $(\pi_1,\pi_2,\dots,\pi_l)$-partitions, that gives 
$$(r-1)^l{n-1 \brace k}_{r-1}^{(l)},$$
from the two ways of counting we get the result.
\end{itemize} 
\end{proof}
\

\begin{rem} Using the previous recurrences it is easy to get the following special cases
\begin{align} 
\label{specS1}&{n \brack r}_r^{(l)}=r^l(r+1)^l\cdots(n-2)^l(n-1)^l=(r^{\overline{n-r}})^l,\qquad \text{for }  n\geq r\\
\intertext{and }
\label{specS2}&{n \brace r}_r^{(l)}=r^{l(n-r)},\qquad \text{for } n\geq r.
\end{align}
\end{rem}
\section{Orthogonality of $(l,r)$-Stirling numbers pair}
\begin{thm}For $n \geq k \geq 0$, for all positive integer $l$, we have the two orthogonality relations bellow
\begin{equation}\label{Orth1}
\sum_{j}{n \brack j}_r^{(l)}{j \brace k}_r^{(l)}(-1)^j=\left\lbrace\begin{array}{l r}
(-1)^n \delta_{n,k}, &\qquad \text{for }  n \geq r;\\
&\\
0, &\qquad \text{for }  n < r
\end{array}\right.	
\end{equation}
and
\begin{equation}\label{Orth2}
\sum_{j}{j \brack n}_r^{(l)}{k \brace j}_r^{(l)}(-1)^j=\left\lbrace\begin{array}{l r}
(-1)^n \delta_{n,k}, &\qquad \text{for }  n \geq r;\\
&\\
0, &\qquad \text{for }  n < r.
\end{array}\right.
\end{equation}
\end{thm}
\begin{proof}
Let us start by Identity \eqref{Orth1}. The proof goes by induction on $n$
\begin{itemize}
\item For $n<r$ the assertion is obvious.
\item For $n=r$,
\begin{equation*}
\begin{split}
\sum_{j}{r \brack j}_r^{(l)}{j \brace k}_r^{(l)}(-1)^j&={r \brace k}_r^{(l)}(-1)^r=(-1)^r\delta_{k,r}.\end{split}
\end{equation*}
\item For $n> r$, Theorem \ref{RecS1} and the induction hypothesis implies that
\begin{equation*}
\begin{split}
\sum_{j}{n \brack j}_r^{(l)}{j \brace k}_r^{(l)}(-1)^j&=\sum_j \left({n-1 \brack j-1}_r^{(l)}+(n-1)^l{n-1 \brack j}_r^{(l)}\right){j \brace k}_r^{(l)}(-1)^j\\
&=(n-1)^l(-1)^{n-1}\delta_{n-1,k}+\sum_j{n-1 \brack j-1}_r^{(l)}{j \brace k}_r^{(l)}(-1)^j,\\
\end{split}	
\end{equation*}
\noindent and from Theorem \ref{RecS2}, we get
\begin{equation*}
\resizebox*{1\linewidth}{!}{$
\begin{split}
\sum_{j}{n \brack j}_r^{(l)}{j \brace k}_r^{(l)}(-1)^j&=(n-1)^l(-1)^{n-1}\delta_{n-1,k}-(-1)^{n-1}\delta_{n-1,k-1}-(k)^l(-1)^{n-1}\delta_{n-1,k}\\
&=(-1)^n\delta_{n,k}.
\end{split}$}
\end{equation*}
\end{itemize}
For the Identity \eqref{Orth2}, we go by induction on $k$ as same as the previous proof.
\end{proof}

\section{Properties via symmetric functions}\label{Sec3}
Let $x_1, x_2, \dots,x_n$ be $n$ random variables. We denote, respectively, by \\$e_k(x_1,x_2,\dots,x_n)$ and $h_k(x_1,x_2,\dots,x_n)$ the elementary symmetric function and the complete homogeneous symmetric function of degree $k$ in $n$-variables given for $n\geq k \geq 1$, by
\begin{equation}\label{esf}
e_k(x_1,x_2,\dots,x_n)=\sum_{1 \leq i_1 < i_2 < \cdots < i_k \leq n}x_{i_1}\cdots x_{i_k} 
\end{equation}
and
\begin{equation}\label{csf}
h_k(x_1,x_2,\dots,x_n)=\sum_{1\leq i_1 \leq i_2 \leq \cdots \leq i_k\leq n}x_{i_1}\cdots x_{i_k}.
\end{equation}

In particular $e_0(x_1,x_2,\dots,x_n)=h_0(x_1,x_2,\dots,x_n)=\delta_{0,n}$.

The generating functions of the symmetric functions are given by
\begin{equation}\label{Gesf}
E(t)=\sum_{k \geq 0}e_k(x_1,x_2,\cdots,x_n)t^k=\prod_{i=1}^n(1+x_it)
\end{equation}
and 
\begin{equation}\label{Gcsf}
H(t)=\sum_{k \geq 0}h_k(x_1,x_2,\cdots,x_n)t^k=\prod_{i=1}^n(1-x_it)^{-1}.
\end{equation}

For more details about symmetric functions we refer readers to \cite{Abr,Mac,Mer} and the references therein.

Let us now give some results linked to the symmetric functions and their generating functions.
\begin{thm}
The $(l,r)$-Stirling of the first kind and the elementary symmetric function are linked as
\label{Sti1ToSym}
\begin{equation}
{n+1 \brack n+1-k}_r^{(l)}=e_k(r^l,\dots,n^l),
\end{equation}
equivalently
\begin{equation}
{n \brack k}_r^{(l)}=e_{n-k}(r^l,\dots,(n-1)^l).
\end{equation}

\end{thm}
\begin{proof}
It is clear that in each $(\sigma_1,\sigma_2,\dots,\sigma_l)$-permutation having $(n-k)$ cycles with $\{1,\dots,r\}$ lead, we have $\{1,2,\dots,r,y_{r+1},\dots,y_{n-k}\}$ lead a cycle and $\{x_1,x_2,\dots,x_k\}$ elements don't lead where $r<y_{r+1}<y_{n-k}<\dots \leq n$  and $r<x_1<x_2<\dots \leq n$.

To construct all $(\sigma_1,\sigma_2,\dots,\sigma_l)$-permutations having $(n-k)$ cycles where $\{1,\dots,r\}$ lead, we proceed as follows
\begin{itemize}

\item Construct $(n-k)$ cycles having only one element from $\{1,2,\dots,r,$ $y_{r+1},\dots,{y_{n-k}}\}$, i. e.
$$\sigma=(1)(2)\dots(r)(y_{r+1})\dots(y_{n-k}),$$
\item Insert $x_1$ after an element of cycles smaller than $x_1$, we have $(x_1-1)$ ways of inserting $x_1$. Then Insert $x_2$ after an element of cycles smaller than $x_2$, we have $(x_2-1)$ choices, and so on. We have $(x_1-1)(x_2-1)\cdots(x_k-1)$ ways to construct a permutation. 
\item Repeat the process with each permutation $\sigma \in \{\sigma_1,\dots,\sigma_l\}$, so we have $(x_1-1)^l(x_2-1)^l\cdots(x_k-1)^l$ ways of construction. 
\item Summing over all possible set of numbers $\{x_1,x_2,\dots, x_k\}$, hence the total number of ways to construct  $(\sigma_1,\sigma_2,\dots,\sigma_l)$-permutations having $(n-k)$ cycles with $\{1,\dots,r\}$ lead is 

\begin{equation*}
\begin{split}
{n \brack n-k}_r^{(l)}&=\sum_{r<x_1<x_2<\cdots\leq n}(x_1-1)^l(x_2-1)^l\cdots(x_k-1)^l\\&=\sum_{r \leq x_1<x_2<\cdots< n}x_1^l x_2^l\cdots x_k^l\\&=e_k(r^l,\dots,(n-1)^l).
\end{split}
\end{equation*}
\end{itemize}
\end{proof}
\begin{thm} \label{Sti2ToSym} The $(l,r)$-Stirling of the first kind and the complete homogeneous symmetric function are linked as
\begin{equation}
{n +k \brace n}_r^{(l)}=h_k(r^l,\dots,n^l),
\end{equation}
\end{thm}
\begin{proof}
Let us count the number of $(\pi_1,\pi_2,\dots,\pi_l)$-partitions of $[n+k]$ into $n$ blocks with $\{1,2,\dots,r\}$ are leaders. First, we denote, $\{y_1,y_2,\dots,y_k\}$ the elements that are not leaders where $y_1<y_2<\dots<y_k$. Let $x_i$ be  the number of leaders smaller than $y_i$, $i \in \{1,\dots,k\}$, it is clear that $r \leq i_1 \leq i_2 \leq \cdots \leq i_k \leq n$. 

The construction of such partition goes as follows
\begin{itemize}
\item Construct a partition of $n$ blocks with $[n+k]$\textbackslash$\{y_1,y_2,\dots,y_k\}$ where ${1,2,\dots,r}$ are leaders, i. e. 
$$\{1\}\{2\}\dots\{r\}\{z_{r+1}\}\dots\{z_{n}\}.$$
\item Insert the $\{y_1,y_2,\dots,y_k\}$ elements to the $n$ blocks. It is clear that $y_i$ can belong only to a block having a leader smaller than $y_i$, we have $x_1\cdot x_2 \cdots x_k $ ways to do.
\item Repeat the process with each partition $\pi \in \{\pi_1,\dots,\pi_l\}$, so we have $(x_1)^l(x_2)^l\dots(x_k)^l$ ways of construction. 
\item Summing over all possible set of numbers $\{x_1,x_2,\dots, x_k\}$, hence the total number of ways to construct  $(\pi_1,\pi_2,\dots,\pi_l)$-partitions of $[n+k]$ having $n$ blocks with $\{1,\dots,r\}$ lead is 
\begin{equation*}
\begin{split}
{n+k \brace n}_r^{(l)}&=\sum_{r \leq x_1\leq x_2\leq \cdots\leq n}x_1^lx_2^l\cdots x_k^l\\&=h_k(r^l,\dots,n^l).
\end{split}
\end{equation*}
\end{itemize}
\end{proof}
\section{Generating functions}
Now, we can use the symmetric functions to construct the generating functions for the $(l,r)$-Stirling of both kinds. 
\begin{thm} The generating function for the $(l,r)$-Stirling numbers of the first kind is
\begin{equation}
\sum_k  {n \brack k}_r^{(l)}z^k=z^r \prod _{i=r}^{n-1} \left(z+i^l\right)=z^r  \left(z+r^l\right)\left(z+(r+1)^l\right)\cdots \left(z+(n-1)^l\right),
\end{equation}
\end{thm}
\begin{proof}
From Theorem \ref{Sti1ToSym} and the generating function \eqref{Gesf} we obtain
\begin{equation}
\begin{split}
\sum_k  {n \brack k}_r^{(l)}z^k&=z^n\sum_{k}e_{k}(r^l,\dots,(n-1)^l)(z^{-1})^k\\
&=z^n \prod_{i=r}^{n-1}\left(1+\frac{i^l}{z}\right)\\
&=z^r\prod_{i=r}^{n-1}(z+i^l).
\end{split}
\end{equation}
\end{proof}
\begin{thm}The generating function for the $(l,r)$-Stirling numbers of the second kind is
\begin{equation}
\sum_{n = k}  {n \brace k}_r^{(l)}z^n=z^k\left(\prod_{i=r}^k(1-z i^l)\right)^{-1}=\frac{z^k}{(1-z r^l)(1-z (r+1)^l)(1-z k^l)}.
\end{equation}
\end{thm}
\begin{proof}
From Theorem \ref{Sti2ToSym} and the generating function of homogeneous symmetric function \eqref{Gcsf}, we obtain
\begin{equation*}
\resizebox*{1\textwidth}{!}{$
\begin{split}
\sum_{n \geq  k} {n \brace k}_r^{(l)}z^n&=\sum_{j \geq 0} {k+j \brace k}z^{k+j}=z^k\sum_{j\geq 0}h_j(r^l,\dots,k^l)z^{j}=z^k\left(\prod_{i=r}^k(1-z i^l)\right)^{-1}.
\end{split}$}
\end{equation*}
\end{proof}

In the following theorem we investigate the symmetric functions to obtain a convolution formula for the $(l,r)$-Stirling numbers of both kinds.
\begin{thm}For all positive integers $l$, $n$, $k$ and $r$ with $(n \geq k \geq r)$, we have
\begin{equation}
 \sum_{ \begin{array}{c} i_0+2i_1\cdots+2^li_{l}=k \\
i_0,\cdots,i_{l} \geq 0
\end{array}}{n+i_l \brace n}_r^{(2^l)}\prod_{s=0}^{l-1}{n+1 \brack n+1-i_s}_r^{(2^s)}={n+k \brace n}_r.
\end{equation}
\end{thm}
\begin{proof}Let us consider the generating function of the complete homogeneous symmetric function \eqref{Gcsf}. From that we have
\begin{equation*}
\resizebox*{1\textwidth}{!}{$
\begin{split}
\sum_{k \geq 0}h_k(x_1,\dots,x_n)z^k&=\prod_{i=1}^n\frac{1}{(1-x_iz)}\\
&=\prod_{i=1}^n\frac{1}{(1-x_iz)}\prod_{s=0}^{l-1}\left(\frac{1+x_i^{2^s}z^{2^s}}{1+x_i^{2^s}z^{2^s}}\right)\\
&=\prod_{i=1}^n\frac{1}{(1-x_i^{2^l}z^{2^l})}\prod_{s=0}^{l-1}\left(1+x_i^{2^s}z^{2^s}\right)\\
&=\sum_{k \geq 0}h_k(x_1^{2^l},\dots,x_n^{2^l})z^{2^lk}\prod_{s=0}^{l-1}\sum_{k \geq 0}e_k(x_1^{2^s},\dots,x_n^{2^s})z^{2^sk}\\
&=\sum_{k \geq 0}h_k(x_1^{2^l},\dots,x_n^{2^l})z^{2^lk}\sum_{k \geq 0}e_k(x_1,\dots,x_n)z^{k}\sum_{k \geq 0}e_k(x_1^{2},\dots,x_n^{2})z^{2k}\cdots \sum_{k \geq 0}e_k(x_1^{2^{l-1}},\dots,x_n^{2^{l-1}})z^{2^{l-1}k}\\
&=\sum_{k \geq 0}\left( \sum_{\begin{array}{c}
i_0+2i_1+\dots+2^li_l=k;\\
i_0,\dots,i_l\geq 0.
\end{array}}h_{i_l}(x_1^{2^l},\dots,x_n^{2^l})\prod_{s=0}^{l-1}e_{i_s}(x_1^{2^s},\dots,x_n^{2^s})\right)z^k.
\end{split}
$}
\end{equation*}
From Theorem \ref{Sti1ToSym} and Theorem \ref{Sti2ToSym} and by comparing the coefficients of $z^k$ of the two sides the result holds true.
\end{proof}
The simplest case of the previous theorem is the corollary bellow which generalize the result of Broder \cite{Bro}.
\begin{cor}
For $l=1$, we have 
\begin{equation}
 \sum_{i=0}^{\lfloor k/2\rfloor}{n+i \brace n}_r^{(2)}{n+1 \brack n+1+2i-k}_r={n+k \brace n}_r.
\end{equation}
\end{cor}
\section{The $(l,r)$-Stirling numbers, the sum powers and Bernoulli polynomials}
Recall, for every integer $n \geq 0$, the Bernoulli polynomials, denoted $B_n(x)$, are defined by
\begin{equation}\label{Bern}
\sum_{n=0}^{\infty}B_n(x)\frac{t^n}{n}=\frac{te^{xt}}{e^t-1}.
\end{equation}
The sum of the powers of natural numbers is closely related to the Bernoulli polynomials $B_n(x)$. Jacobi \cite{Jac,Sri} gives the following identity using the sum of powers and Bernoulli polynomials
\begin{equation}\label{JacId}
\sum_{j=1}^{n}j^m=\frac{B_{m+1}(n+1)-B_{m+1}(0)}{m+1}.
\end{equation}

The following theorem gives the relation between $(l,r)$-Stirling of both kinds and Bernoulli polynomials.  
\begin{thm}For all positive integers $n$, $k$ and $l$, we have
\begin{equation}
\sum_{j=0}^k(-1)^j(j+1){n+1\brack n-j}^{(l)}{n+k-j\brace n}^{(l)}=\frac{B_{lk+l+1}(n+1)-B_{lk+l+1}(0)}{lk+l+1},
\end{equation}
\end{thm}
\begin{proof}
In the first hand we have Jacobi's Identity \eqref{JacId}
\begin{equation}\label{Faul}
\sum_{j=1}^{n}(j^l)^k=\frac{B_{lk+1}(n+1)-B_{lk+1}(0)}{lk+1},
\end{equation}
in the second hand, we have
$$H(t)=\sum_{k\geq 0}h_k(1^l,2^l,\dots ,n^l)t^k=\prod_{j=1}^n\frac{1}{(1-j^st)}$$
and
$$E(t)=\sum_{k\geq 0}e_k(1^l,2^l,\dots ,n^l)t^k=\prod_{j=1}^n(1+j^st),$$
from the obvious observation that $H(t)=1/E(-t)$, we obtain
\begin{equation}\label{fside}
\frac{d}{dt}\ln{H(t)}=\frac{H'(t)}{H(t)}=H(t)E'(-t)
\end{equation}
but
\begin{equation}\label{sside}
\frac{d}{dt}\ln{H(t)}=\sum_{j=1}^n \frac{j^l}{(1-j^lt)}=\sum_{k \geq 0}\sum_{j=1}^n j^{s(k+1)}t^k.
\end{equation}
Then from equations \eqref{fside} and \eqref{sside}, we get
\begin{equation}
\begin{split}
\sum_{k \geq 0}\sum_{j=1}^n j^{s(k+1)}t^k&=H(t)E'(-t)\\
&=\left(\sum_{k\geq 0} h_k(1^l,\dots,n^l)t^{k}\right)\left(\sum_{k\geq 1}k (-1)^{k-1} e_k(1^l,\dots,n^l)t^{k-1}\right).\\
\end{split}
\end{equation}
Cauchy product and equating coefficient of $t^k$ gives
\begin{equation}
\begin{split}
\sum_{j=1}^n j^{s(k+1)}&=\sum_{j\geq 1}^n (j+1) (-1)^{j}e_{j+1}(1^l,\dots,n^l)h_{k-j}(1^l,\dots,n^l),
\end{split}
\end{equation}
replacing symmetric functions by stirling numbers from Theorem \ref{Sti1ToSym} and Theorem \ref{Sti2ToSym}, and comparing with Equation \eqref{Faul} we get the result.

\end{proof}
\section{Multiple zeta function and $(l,r)$-Stirling numbers of the first kind}\label{Sec4}
\noindent

For any ordered sequence of positive integers $i_1,i_2,\dots,i_k$, the \textit{multiple zeta function} is introduced by Hoffman \cite{Hof}
and independently Zagier \cite{Zag} by the following infinite sums
\begin{equation}\label{MultZ}
\zeta(i_1,i_2,\dots,i_k)=\sum_{0< j_1 < j_2 < \cdots <j_k}\frac{1}{j_1^{i_1}j_2^{i_2}\cdots j_k^{i_k}}.
\end{equation}

Recently, the multiple zeta function has been studied quite intensively by many authors in various fields of mathematics and physics (see \cite{Ara,Bra,Bor,Hof,Zag,Zud}).
Here we give a relation between $(l,r)$-Stirling numbers of the first kind and the multiple zeta function.

\begin{thm}\label{StirNest}For all positive integers $n$, $k$, $l$ and $r$ with $(n \geq k \geq r)$, we have
\begin{equation}
\begin{split}
{n+1  \brack k+1}_r^{(l)}&=\left(\frac{n!}{(r-1)!}\right)^l\sum_{j_{k}=k}^{n}\sum_{j_{k-1}=k-1}^{j_{k}-1}\cdots \sum_{j_{r}=r}^{j_{(r+1)}-1}\frac{1}{\left(j_{r}j_2\cdots j_{k} \right)^l}\\&=\left(\frac{n!}{(r-1)!}\right)^l\sum_{r-1< j_1 < j_2 < \cdots <j_k\leq n}\frac{1}{(j_1j_2\cdots j_k)^{l}}.
\end{split}
\end{equation}
\end{thm}
\begin{proof}
Since $ {n \brack k}_r^{(l)}={n-1 \brack k-1}_r^{(l)}+(n-1)^l{n-1 \brack k}_r^{(l)}$ from Theorem \ref{RecS1}. If we proceed iteratively, we obtain that
\begin{equation}\label{it}
{n \brack k}_r^{(l)}=\left((n-1)!\right)^l\sum_{j=k-1}^{n-1}\frac{1}{(j!)^l}{j \brack k-1}_r^{(l)}.
\end{equation}
For $k={r}$, from \eqref{specS1} and \eqref{it} we obtain
\begin{equation}\label{itr}
{n \brack r}_r^{(l)}=(r^{\overline{n-r}})^l=\left(\frac{(n-1)!}{(r-1)!}\right)^l.
\end{equation}
For $k={r+1}$, from \eqref{it} and \eqref{itr} we obtain
\begin{equation}\label{itr1}
\begin{split}
{n \brack r+1}_r^{(l)}&=\left((n-1)!\right)^l\sum_{j=r}^{n-1}\frac{1}{(j!)^l}{j \brack r}_r^{(l)}\\
&=\left(\frac{(n-1)!}{(r-1)!}\right)^l\sum_{j=r}^{n-1}\left(\frac{(j-1)!}{j!}\right)^l\\
&=\left(\frac{(n-1)!}{(r-1)!}\right)^l\sum_{j=r}^{n-1}\frac{1}{j^l}.
\end{split}
\end{equation}
For $k=r+2$, from \eqref{itr} and \eqref{itr1} we obtain
\begin{equation}
{n \brack r+2}_r^{(l)}=\left(\frac{(n-1)!}{(r-1)!}\right)^l\sum_{j=r+1}^{n-1}\sum_{i=r}^{j-1}\frac{1}{(ij)^l}, 
\end{equation}
iterating the process with $k\in \{r+3,r+4,\dots\}$ and so on, then yields the result.
\end{proof}
\begin{prop}\label{StirZeta}For $r=1$, we have
\begin{equation}
\lim_{n \to \infty} \frac{1}{\left(n!\right)^l}{\displaystyle {n+1 \brack k+1}}^{(l)}=\zeta(\{l\}_k),
\end{equation}
where $\{l\}_n=(\underbrace{l,l,\dots,l}_{n \text{ times}}).$
\end{prop}
\begin{proof}
The proposition follows immediately from the definition of multiple zeta function \eqref{MultZ} as an infinity sums and Theorem \ref{StirNest} for $r=1$.
\end{proof}

\begin{cor}For $k \geq 1$, we have
\begin{itemize}
\item For $l=2$
\begin{equation}\label{zetal2}
\lim_{n \to \infty} \frac{1}{\left(n!\right)^2}\displaystyle {n+1 \brack k+1}^{(2)}=\frac{\pi^{2k}}{(2k+1)!}.
\end{equation}
\item For $l=4$
\begin{equation}\label{zetal4}
\lim_{n \to \infty} \frac{1}{\left(n!\right)^4}\displaystyle {n+1 \brack k+1}^{(4)}=\frac{4(2\pi)^{4 k}}{(4 k+2)!}\left(\frac{1}{2}\right)^{2k+1}.
\end{equation}
\item For $l=6$
\begin{equation}\label{zetal6}
\lim_{n \to \infty} \frac{1}{\left(n!\right)^6}\displaystyle {n+1 \brack k+1}^{(6)}=\frac{6(2\pi)^{6 k}}{(6 k+3)!}.
\end{equation}
\item For $l=8$
\begin{equation}\label{zetal8}
\lim_{n \to \infty} \frac{1}{\left(n!\right)^8}\displaystyle {n+1 \brack k+1}^{(8)}=\frac{\pi ^{8 k}}{(8 k+4)!}2^{8 k+3} \left(\left(1+\frac{1}{\sqrt{2}}\right)^{4k+2}+\left(1-\frac{1}{\sqrt{2}}\right)^{4 k+2}\right).
\end{equation}
\end{itemize}
\end{cor}
\begin{proof}
Authors in \cite{Bor} give the following special values of multiple zeta function
\begin{eqnarray*}
\zeta(\{2\}_n)&=&\frac{\pi^{2n}}{(2n+1)!},\\
\zeta(\{4\}_n)&=&\frac{4(2\pi)^{4 n}}{(4 n+2)!}\left(\frac{1}{2}\right)^{2n+1},\\
\zeta(\{6\}_n)&=&\frac{6(2\pi)^{6 n}}{(6 n+3)!},\\
\zeta(\{8\}_n)&=&\frac{\pi ^{8 n}}{(8 n+4)!}2^{8 n+3} \left(\left(1+\frac{1}{\sqrt{2}}\right)^{4
   n+2}+\left(1-\frac{1}{\sqrt{2}}\right)^{4 n+2}\right),\\
\end{eqnarray*}
the corollary is a consequence of the previous special cases and Proposition \ref{StirZeta}. 
\end{proof}

\section{Remarks}\label{Sec6}
\begin{itemize}
\item The $(l,r)$-Stirling gives another graphical view of Rooks polynomials of higher dimensions in triangle boards \cite{Kr,Zin} using set partitions.
\item In this work we gives a limit representation of multiple zeta function using $(l,r)$-Stirling numbers. 
\item We can obtain the well-known Euler identity $\zeta(2)=\frac{\pi^2}{6}$ from Equation \eqref{zetal2} for $k=1$.
\end{itemize}
\subsection*{Acknowledgment}
We would like to thank the anonymous reviewers for their suggestions and comments
which improved the quality of the present paper. The paper was partially
supported by the DGRSDT grant \textnumero:C0656701.


\begin{thebibliography}{1}
  \bibitem{Abl}Ablinger, J., and Blümlein, J. (2013). Harmonic sums, polylogarithms, special numbers, and their generalizations. In Computer Algebra in Quantum Field Theory (pp. 1-32). Springer, Vienna.
  \bibitem{Abr}Abramowitz, M., and Stegun, I. A. (1964). Handbook of mathematical functions with formulas, graphs, and mathematical tables (Vol. 55). US Government printing office.
  \bibitem{Kr}Alayont, F., and Krzywonos, N. (2013). Rook polynomials in three and higher dimensions. Involve, a Journal of Mathematics, 6(1), 35-52.
  \bibitem{Ara}Arakawa, T., and Kaneko, M. (1999). Multiple zeta values, poly-Bernoulli numbers, and related zeta functions. Nagoya Mathematical Journal, 153, 189-209.
  \bibitem{Bra}Bradley, D. M. (2005). Partition identities for the multiple zeta function. In Zeta functions, topology and quantum physics (pp. 19-29). Springer, Boston, MA.
  \bibitem{Bro}Broder, A. Z. (1984). The r-Stirling numbers. Discrete Mathematics, 49(3), 241-259.
  \bibitem{Bor}Borwein, J. M., and Bradley, D. M. (1997). Evaluations of k-fold Euler/Zagier sums: a compendium of results for arbitrary k. the electronic journal of combinatorics, 4(2), R5.
  \bibitem{Dum}Dumont, D. (1974). Interprétations combinatoires des nombres de Genocchi. Duke Mathematical Journal, 41(2), 305-318.
  \bibitem{Foa}Foata, D., and Han, G. N. (2000). Principes de combinatoire classique. Lecture notes, Strasbourg.
  \bibitem{Gel}Gelineau, Y., and Zeng, J. (2010). Combinatorial interpretations of the Jacobi-Stirling numbers. the electronic journal of combinatorics, 17(R70), 1.
  \bibitem{Hof}Hoffman, M. (1992). Multiple harmonic series. Pacific Journal of Mathematics, 152(2), 275-290.
  \bibitem{Jac}Jacobi, C. G. J. (1834). De usu legitimo formulae summatoriae Maclaurinianae. Journal für die reine und angewandte Mathematik, 1834(12), 263-272.
  \bibitem{Mac}Macdonald, I. G. (1998). Symmetric functions and Hall polynomials. Oxford university press.
  \bibitem{Mer1}Merca, M., and Cuza, A. I. (2012). A special case of the generalized Girard-Waring formula. Journal of Integer Sequences, 15(2), 3.
  \bibitem{Mer}Merca, M. (2016). New convolutions for complete and elementary symmetric functions. Integral Transforms and Special Functions, 27(12), 965-973.
  \bibitem{Sri}Srivastava, H. M. (2000, July). Some formulas for the Bernoulli and Euler polynomials at rational arguments. In Mathematical Proceedings of the Cambridge Philosophical Society (Vol. 129, No. 1, pp. 77-84).
  \bibitem{Zag}Zagier, D. (1994). Values of zeta functions and their applications. In First European Congress of Mathematics Paris, July 6–10, 1992 (pp. 497-512). Birkhäuser Basel.
  \bibitem{Zin}Zindle, B. (2007). Rook polynomials for chessboards of two and three dimensions. Master thesis.
  \bibitem{Zud}Zudilin, W. W. (2003). Algebraic relations for multiple zeta values. Russian Mathematical Surveys, 58(1), 1-29.
\end{thebibliography}
\end{document}